\documentclass[11pt, leqno]{article}
\usepackage[T1]{fontenc}
\usepackage[ansinew]{inputenc}
\usepackage{latexsym}
\usepackage{amsfonts}
\usepackage{amssymb}
\usepackage{amsmath,mathrsfs,amsthm,amscd}
\usepackage{url}
\usepackage{hyperref}
\usepackage{paralist}
\usepackage{verbatim}
\usepackage[vcentermath]{youngtab}
\newtheorem{theo}{Theorem}[section]
\newtheorem{lem}[theo]{Lemma}
\newtheorem{prop}[theo]{Proposition}
\newtheorem{coro}[theo]{Corollary}
\theoremstyle{definition}
\newtheorem{defi}[theo]{Definition}
\newtheorem{exe}[theo]{Example}
\theoremstyle{remark}
\newtheorem{rem}[theo]{Remark}
\newtheorem{note}[theo]{Note}

\newcommand{\w}{\omega}
\newcommand{\g}{\gamma}
\newcommand{\al}{\alpha}
\newcommand{\la}{\lambda}

\newcommand{\be}{\beta}

\newcommand{\N}{\mathbb{N}}

\def\thebibliography#1{\section*{{\large REFERENCES}\markboth
 {REFERENCES}{REFERENCES}}\list
 {[\arabic{enumi}]}{\settowidth\labelwidth{[#1]}\leftmargin\labelwidth
 \advance\leftmargin\labelsep
 \usecounter{enumi}}
 \def\newblock{\hskip .11em plus .33em minus -.07em}
 \sloppy
 \sfcode`\.=1000\relax}

\def\qed{\relax\ifmmode\hskip2em \Box\else\unskip\nobreak\hskip1em \hfill$\Box$\fi}

\title{Derivations and dimensionally nilpotent derivations in Lie triple algebras}

\author{Abdoulaye Dembega\thanks{\texttt{doulaydem@yahoo.fr}}\\
	Universit\'{e} Norbert Zongo\\
	BP 376 Koudougou, Burkina Faso
\and	
	Amidou Konkobo\thanks{\texttt{konkoboa@yahoo.fr}} and Moussa Ouattara\thanks{\texttt{ouatt\_ken@yahoo.fr}}\\
	Universit\'e Joseph KI-ZERBO\\
	03 BP 7021 Ouagadougou 03,
	Burkina Faso
}
\date{\today}

\begin{document}
\maketitle
\begin{abstract}

In this paper, we first study derivations in non nilpotent Lie triple algebras. We determine the structure of derivation algebra according to whether the algebra admits an idempotent or a pseudo-idempotent. We study the multiplicative structure of non nilpotent dimensionally nilpotent Lie triple algebras. We show that when $n=2p+1$   the adapted basis coincides with the canonical basis of the gametic algebra $G(2p+2,2)$ or this one obviously associated to a pseudo-idempotent and if $n=2p$ then the algebra is either one of the precedent case or a conservative Bernstein algebra.



\end{abstract}
\bigskip\bigskip

\section{Introduction}
A $n+1$ finite dimensional algebra $A$ is dimensionally nilpotent if there is a derivation $d: A\longrightarrow A$ such that $d^{n+1}=0$ and
$d^{n}\neq 0$. This notion has been studied by G.F. Leger and P.L. Manley\cite{Leger} for Lie algebras, J.M. Osborn \cite{Osborn} for Jordan algebras, Micali and Ouattara\cite{Micali} for genetic algebras. Recently, V. Eberlin \cite{Eberlin} has deepened the work of the authors of \cite{Leger} in his thesis. Regarding Jordan algebras, Osborn shows that every dimensionally nilpotent Jordan $K$-algebra is either nilpotent or satisfies
$A/Rad(A)\simeq K$.

We study the case of non nilpotent  dimensionally Lie triple algebras. In an adapted basis we caracterize the multiplicative structure of these algebras relative to the parity of $n$. More precisely we show that when $n=2p+1$, the adapted basis coincides with canonical basis of the gametic algebra $G(2p+2,2)$ or this one obviously associated to a pseudo-idempotent. If $n=2p$ then this algebra is either one of the precedent case or a train algebra of rank $3$ which is a Jordan algebra \cite{Ouat}. Since Jordan algebras are also Lie triple ones the final corollary describes non nilpotent dimensionally nilpotent Jordan algebras.

\section{Preliminaries}

 A \textit{Lie triple algebra} is a commutative algebra satisfying
 \begin{eqnarray}\label{Eq1}
 2x(x(xy)) + yx^3 = 3x(yx^2)
 \end{eqnarray}
while a \textit{Jordan algebra} is a commutative algebra satisfying
\begin{eqnarray*}
x^2(yx)=(x^2y)x.
\end{eqnarray*}

 Every Jordan algebra satisfies identity \eqref{Eq1}.

\begin{theo}[\cite{Hentzel}]
Let $A$ be a Lie triple algebra and  $L$ the ideal generated by the associators
  $(x^2, x, x)$. Then $L^2 = 0$ and $A/L$ is a Jordan algebra.
\end{theo}

\begin{defi}
A \textit{pseudo-idempotent} of $A$ is a non-zero element $e$ such that there is $t \ne 0$ in $L$ satisfying $e^2 = e + t$ and $et = \frac{1}{2}t$.
\end{defi}

\begin{theo}[\cite{BKO1}]\label{idemp}
Every Lie triple non nilalgebra contains either a non-zero idempotent, or a pseudo-idempotent.
\end{theo}

\begin{defi}
	An ideal $I$ of an algebra $A$ is said to be  \textit{caracteristic} if $d(I)\subseteq I$ for every derivation $d$ of $A$.
An ideal $I$ of an algebra $A$ is said to be  \textit{d-invariant} if $d(I)\subseteq I$ for a given derivation $d$ of $A$.\end{defi}

\section{Caracterization of derivations}

In this paragraph we study the derivations in Lie triple non nilalgebras . We give a caracterization, distinguishing two cases: with an idempotent or with a pseudo-idempotent.

\subsection{ Lie triple algebras with idempotent}
Relative to the non-zero idempotent $e$, $A$ admits the following Peirce decomposition $A=A_e(1)\oplus A_e(\frac{1}{2})\oplus A_e(0)$. Relations between Peirce components and the products of their elements are ruled by the following lemma:

\begin{lem}[\cite{BKO1}, Lemme 2.2]\label{decomp1}
	Let $A=A_e(1)\oplus A_e(1/2)\oplus A_e(0)$ be the Peirce decomposition of $A$ relative to a non-zero idempotent. Then\\
	\begin{compactenum}[$(i)$]
		\item $A_e(1/2)A_e(1/2)\subseteq A_e(1)+A_e(0)$, $A_e(\la)A_e(\la)\subseteq A_e(\la)$, $A_e(\la)A_e(1/2)\subseteq A_e(1/2)$,\\
		$A_e(\la)A_e(1-\la)=0$, $(\la=0,1)$ ;
		\item $(x_1y_1)a_{1/2} =x_1(y_1a_{1/2})+y_1(x_1a_{1/2})$,\\
		$(x_0y_0)a_{1/2} =x_0(y_0a_{1/2})+y_0(x_0a_{1/2})$ ;
		\item $[x_1(x_{1/2}a_{1/2})]_1=[(x_1x_{1/2})a_{1/2}+(x_1a_{1/2})x_{1/2}]_1$,\\
		$[x_0(x_{1/2}a_{1/2})]_0=[(x_0x_{1/2})a_{1/2}+(x_0a_{1/2})x_{1/2}]_0$ ;
		\item $[(x_1x_{1/2})y_{1/2}]_0=[(x_1y_{1/2})x_{1/2}]_0$,\\
		$[(x_0x_{1/2})y_{1/2}]_1=[(x_0y_{1/2})x_{1/2}]_1$ ;
		\item $x_0(y_1a_{1/2})=y_1(x_0a_{1/2})$ ;
		\item $x_{1/2}(x_{1/2}^2)_1=x_{1/2}(x_{1/2}^2)_0=\frac12 x_{1/2}^3$ ;
		\item $(x_{1/2}y_{1/2})_0z_{1/2} + (y_{1/2}z_{1/2})_0x_{1/2} + (z_{1/2}x_{1/2})_0y_{1/2}\\
		{\qquad\qquad\qquad}= (x_{1/2}y_{1/2})_1z_{1/2} + (y_{1/2}z_{1/2})_1x_{1/2} + (z_{1/2}x_{1/2})_1y_{1/2}$.
	\end{compactenum}
\end{lem}

Since $A$ is $e$-stable, i.e. $A_e(\la)A_e(1/2)\subseteq A_e(1/2)$ and $[(x_\la x_{1/2})y_{1/2}]_{1-\la}=[(x_\la y_{1/2})x_{1/2}]_{1-\la}$ with $\la=0,1$, calculations on derivations give results similar to \cite[Corollary~2]{BCMO}, precisely.

\begin{theo}\label{Der} Every derivation $d$ of $A$ is determined and only defined by a quadruplet
$(d(e),f_d,g_d,h_d)$ with $f_d \in End_K(A_e(1/2))$, $g_d\in Der_K(A_e(0))$ and $h_d\in Der_K(A_e(1))$  satisfying the following conditions:
\begin{compactenum}[$(i)$]
  \item $d(e)\in A_e(1/2)$ ;
  \item $d(x_1)=h_d(x_1)+2d(e)x_1$ ;
  \item $d(x_{1/2})=f_d(x_{1/2})+2(d(e)x_{1/2})_0-2(d(e)x_{1/2})_1$ ;
  \item $d(x_0)=g_d(x_0)-2d(e)x_0$ ;
   \item $h_d(x_1y_1)=h_d(x_1)y_1+x_1h_d(y_1)$ ;
   \item $g_d(x_0y_0)=g_d(x_0)y_0+x_0g_d(y_0)$ ;
  \item $h_d((x_{1/2}y_{1/2})_1)=[f_d(x_{1/2})y_{1/2}+x_{1/2}f_d(y_{1/2})]_1$ ;
  \item $g_d((x_{1/2}y_{1/2})_0)=[f_d(x_{1/2})y_{1/2}+x_{1/2}f_d(y_{1/2})]_0$ ;
  \item $f_d(x_1x_{1/2})=h_d(x_1)x_{1/2}+x_1f_d(x_{1/2})$ ;
  \item $f_d(x_0x_{1/2})=g_d(x_0)x_{1/2}+x_0f_d(x_{1/2}).$
\end{compactenum}
\end{theo}

\begin{prop}\label{J caract1} Let $A$ be a Lie triple algebra and $A=A_e(1)\oplus A_e(1/2)\oplus A_e(0)$ the Peirce decomposition of  $A$ relative to an idempotent $e\ne 0$.
Subspaces $J_\la=\{x_\la\in A_e(\la) \mid x_\la A_e(1/2)=0\}$ $(\la=0,1)$ and $J=J_0\oplus J_1$ are caracteristic ideals of $A$ and the quotient algebra $A/J$ is a Jordan algebra.
\end{prop}

\begin{proof} Considering $J_\lambda=\ker(S_\lambda)$, with $S_\lambda : A_e(\la) \rightarrow End_K(A_e(1/2)), x_\la \mapsto S_\la(x_\la)$ and $S_\la(x_\la) : a_{1/2}\mapsto x_\la a_{1/2}$. We know by (\cite{Osb}) that $J_\la$ is an ideal of $A_e(\la)$ ($\la=0,1$) and since $A_e(\la)A_e(1/2)\subseteq A_e(1/2)$, then $J=J_1+J_0$ is an ideal of $A$ such that $A/J$ is a Jordan algebra (\cite{O}, Proposition~6.7).
	
Let's consider $d\in Der_K(A)$, $x_{\la}\in J_{\la}(e)$ and $a_{1/2} \in A_e(1/2)$.
We have $0=d(x_{\la}a_{1/2})=x_{\la}d(a_{1/2})+d(x_{\la})a_{1/2}$. But $d(a_{1/2})=f_d(a_{1/2})+2(d(e)a_{1/2})_0-2(d(e)a_{1/2})_1,$ therefore we have  $x_\la d(a_{1/2})=0 $ because
$x_\la(d(e)a_{1/2})_\la=[(x_\la d(e))a_{1/2}+(x_\la a_{1/2})d(e)]_\la$. Hence $d(x_{\la})a_{1/2}=0$ with $\la=0,1.$ But, on the one hand we have $d(x_1)=h_d(x_1)-2d(e)x_1,$ and $0=d(x_1)a_{1/2}= h_d(x_1)a_{1/2}$ and then  $h_d(x_1)\in J_1$, on the other hand we have $d(x_0)=g_d(x_0)-2d(e)x_0$, with $0=d(x_0)a_{1/2}=g_d(x_0)a_{1/2}$ and then  $g_d(x_0)\in J_0$. Hence $d(J_\la)\subseteq J_\la$ and we conclude that $d(J)\subseteq J$.
\end{proof}

\subsection{Lie triple algebras with pseudo-idempotent}

\begin{lem}[\cite{BKO1}, Proposition 4.3]\label{decomp2}
	Let $L=L_e(1)\oplus L_e(1/2)\oplus L_e(0)$ and $A=A_e(1)\oplus A_e(1/2)\oplus A_e(0)$ be the respective Peirce decomposition of $L$ and $A$, relative to the pseudo-idempotent $e$, satisfying $e^2=e+t$ with $t\in L_{1/2}$ fixed. Then
	\begin{compactenum}[$(i)$]
		\item  $A_e(0)L_e(1/2) \subseteq L_e(1/2)$, $A_e(1)L_e(1/2) \subseteq L_e(1/2)$, $A_e(1)L_e(1)\subseteq L_e(1)$, \\
		$A_e(0)L_e(0) \subseteq L_e(0)$, $A_e(0)L_e(1)=A_e(1)L_e(0)=0$, \\
		$A_e(1/2)L_e(0)=A_e(1/2)L_e(1)=A_e(1/2)L_e(1/2)=0$ ;
		\item  $A_e(1)A_e(0)\subseteq L_e(1/2)$, $A_e(0)A_e(1/2)\subseteq A_e(1/2)$, $A_e(1)A_e(1/2)\subseteq A_e(1/2)$,\\
		$A_e(0)A_e(0)\subseteq A_e(0)+L_e(1/2)$, $A_e(1)A_e(1)\subseteq A_e(1)+L_e(1/2)$,\\ $A_e(1/2)A_e(1/2)\subseteq A_e(1)+A_e(0)$ ;
		\item   $(x_0y_0)_{1/2}= 4(x_0t)y_0=4(y_0t)x_0$ ;\\
		$(x_1y_1)_{1/2}= 4(x_1t)y_1=4(y_1t)x_1$ ; \\
		$(x_0y_1)_{1/2}= 4(x_0t)y_1= 4(y_1t)x_0$ ;
		\item $(x_1y_1)a_{1/2} =x_1(y_1a_{1/2})+y_1(x_1a_{1/2})$ ;
		\item $(x_0y_0)a_{1/2} =x_0(y_0a_{1/2})+y_0(x_0a_{1/2})$ ;
		\item $x_0(y_1a_{1/2})=y_1(x_0a_{1/2})$ ;
		\item $[x_0(x_{1/2}a_{1/2})]_0=[(x_0x_{1/2})a_{1/2}+(x_0a_{1/2})x_{1/2}]_0$ ;
		\item $[x_1(x_{1/2}a_{1/2})]_1=[(x_1x_{1/2})a_{1/2}+(x_1a_{1/2})x_{1/2}]_1$ ;
		\item $[(x_0x_{1/2})y_{1/2}]_1=[(x_0y_{1/2})x_{1/2}]_1$ ;\\
		$[(x_1x_{1/2})y_{1/2}]_0=[(x_1y_{1/2})x_{1/2}]_0$ ;
		\item $(x_{1/2}y_{1/2})_0z_{1/2}+(y_{1/2}z_{1/2})_0x_{1/2}+(z_{1/2}x_{1/2})_0y_{1/2}
		\\{\qquad\qquad\qquad}=(x_{1/2}y_{1/2})_1z_{1/2}+(y_{1/2}z_{1/2})_1x_{1/2}+(z_{1/2}x_{1/2})_1y_{1/2}$.
	\end{compactenum}
\end{lem}

\begin{lem}\label{d(t)}
 Let $A$ be a Lie triple algebra and $e$ a pseudo-idempotent of $A$: $e^2=e+t,\; et=\frac12 t,\; t^2=0$ with $t\in L$. For every derivation $d$ of $A$, we have
$$d(t)=0\text{ and } d(e)\in A_e(1/2).$$
\end{lem}

\begin{proof}
Let's consider $d\in Der_K(A)$. Since $e^2=e+t$, we have $2ed(e)=d(e)+d(t)$. Setting $d(e)=[d(e)]_1+[d(e)]_{1/2}+[d(e)]_0$, we have $d(t)=[d(e)]_1-[d(e)]_0$. Because of $2et=t$, we deduce $2ed(t)+2d(e)t=d(t)$. We have $2d(e)t=-[d(e)]_1-[d(e)]_0$. We know that $t\in L_e({1/2})$ and  $L_e({1/2})$ is an ideal of $A$. It follows that $[d(e)]_1=[d(e)]_0=0$.
\end{proof}


\begin{theo}\label{Der pseudo} Every derivation $d$ of $A$ is determined and only defined by a quadruplet $(d(e),f_d,g_d,h_d)$ with $f_d\in End_K(A_e(1/2))$, $g_d\in End_K(A_e(0))$ and $h_d\in End_K(A_e(1))$  satisfying the following conditions:
\begin{compactenum}[$(i)$]
  \item $d(e)\in A_e(1/2)$ ;
  \item $d(x_1)=h_d(x_1)+2d(e)x_1$ ;
  \item $d(x_{1/2})=f_d(x_{1/2})+2(d(e)x_{1/2})_0-2(d(e)x_{1/2})_1$ ;
  \item $d(x_0)=g_d(x_0)-2d(e)x_0$ ;
  \item $h_d((x_1y_1)_1)=[h_d(x_1)y_1+x_1h_d(y_1)]_{1}$ ;\\
  $f_d((x_1y_1)_{1/2})=[h_d(x_1)y_1+x_1h_d(y_1)]_{1/2}=2h_d((x_1y_1)_1)t$ ;
  \item $g_d((x_0y_0)_0)=[g_d(x_0)y_0+x_0g_d(y_0)]_{0}$ ;\\
  $f_d((x_0y_0)_{1/2})=[g_d(x_0)y_0+x_0g_d(y_0)]_{1/2}=2g_d((x_0y_0)_0)t$ ;
  \item $h_d((x_{1/2}y_{1/2})_1)=[f_d(x_{1/2})y_{1/2}+x_{1/2}f_d(y_{1/2})]_1$ ;\\
  		$g_d((x_{1/2}y_{1/2})_0)=[f_d(x_{1/2})y_{1/2}+x_{1/2}f_d(y_{1/2})]_0$ ;
  \item $f_d(x_1x_0)=h_d(x_1)x_0+x_1g_d(x_0)$ ;
  \item $f_d(x_1x_{1/2})=h_d(x_1)x_{1/2}+x_1f_d(x_{1/2})$ ;
  \item $f_d(x_0x_{1/2})=g_d(x_0)x_{1/2}+x_0f_d(x_{1/2}).$
\end{compactenum}
\end{theo}

\begin{proof} Let $d$ be a derivation of $A$ and $e$ a pseudo-idempotent of $A$. Since $d(e)\in A_e(1/2)$, we have $(i)$.
Let $x_1\in A_e(1)$. We have $ex_1=x_1$, and then $d(e)x_1+ed(x_1)=d(x_1)$. Let's set $d(x_1)=a_1+a_{1/2}+a_0$. Then $d(e)x_1+a_1+\frac 12 a_{1/2}=a_1+a_{1/2}+a_0$, and we have $a_{1/2}=2d(e)x_1$ and $a_0=0$. Hence
$d(x_1)=h_d(x_1)+2d(e)x_1$ with $h_d$ an endomorphism of $A_e(1)$ and $(ii)$ is prooved.

By similar calculations we have $(iii)$ and $(iv)$.

with $x_1$, $y_1\in A_e(1)$, we have $(x_1y_1)_{1/2}=4x_1(y_1t)=4y_1(x_1t)=2(x_1y_1)t$
\begin{eqnarray*}
  d(x_1y_1) &=& d[(x_1y_1)_1]+d[(x_1y_1)_{1/2}] =d[(x_1y_1)_1]+2d((x_1y_1)t)\\
   &=&h_d[(x_1y_1)_1]+2d(e)(x_1y_1)+2d((x_1y_1))t+2(x_1y_1)d(t)\\
   &=&h_d[(x_1y_1)_1]+2h_d(x_1y_1)t+2d(e)(x_1y_1),
\end{eqnarray*}
because $d((x_1y_1))t=h_d(x_1y_1)t$.

 We also have,
\begin{eqnarray*}
  d(x_1y_1) &=& d(x_1)y_1+x_1d(y_1)\\
   &=& x_1[h_d(y_1)+2d(e)y_1]+[h_d(x_1)+2d(e)x_1]y_1\\
   &=& h_d(x_1)y_1+x_1h_d(y_1)+2[d(e)y_1]x_1+2[d(e)x_1]y_1.\\
   &=& h_d(x_1)y_1+x_1h_d(y_1)+2d(e)(x_1y_1).
\end{eqnarray*}

It follows that

$$h_d((x_1y_1)_1)=[h_d(x_1)y_1+x_1h_d(y_1)]_{1} \text{ et }$$
$$f_d((x_1y_1)_{1/2})=[h_d(x_1)y_1+x_1h_d(y_1)]_{1/2}=2h_d(x_1y_1)t, \textrm{ and we have } (v).$$

We show by similar calculations that:

$$g_d((x_0y_0)_0)=[g_d(x_0)y_0+x_0g_d(y_0)]_{0}\text{ et } $$
$$f_d((x_0y_0)_{1/2})=[g_d(x_0)y_0+x_0g_d(y_0)]_{1/2}=2g_d(x_0y_0)t, \textrm{and we have } (vi).$$

Let $x_{1/2}, y_{1/2}\in A_e(1/2)$. We have
\begin{eqnarray*}
  d(x_{1/2}y_{1/2}) &=& d((x_{1/2}y_{1/2})_1) +d((x_{1/2}y_{1/2})_0)\\
                    &=& h_d((x_{1/2}y_{1/2})_1)+2d(e)(x_{1/2}y_{1/2})_1\\
                    & &+g_d((x_{1/2}y_{1/2})_0)-2d(e)(x_{1/2}y_{1/2})_0,
  \end{eqnarray*}
  But
  \begin{eqnarray*}
   d(x_{1/2}y_{1/2})&=& d(x_{1/2})y_{1/2}+x_{1/2}d(y_{1/2}) \\
                    &=& (f_d(x_{1/2})+2[d(e)x_{1/2}]_0-2[d(e)x_{1/2}]_1)y_{1/2} \\
                    & & +x_{1/2}(f_d(y_{1/2})+2[d(e)y_{1/2}]_0-2[d(e)y_{1/2}]_1)\\
                    &=&f_d(x_{1/2})y_{1/2}+2y_{1/2}[d(e)x_{1/2}]_0-2y_{1/2}[d(e)x_{1/2}]_1\\
                    & &+x_{1/2}f_d(y_{1/2})+2x_{1/2}[d(e)y_{1/2}]_0-2x_{1/2}[d(e)y_{1/2}]_1\\
                    &=&f_d(x_{1/2})y_{1/2}+x_{1/2}f_d(y_{1/2})+2d(e)(x_{1/2}y_{1/2})_1\\
                    & &-2d(e)(x_{1/2}y_{1/2})_0
   \end{eqnarray*}
because of identity $(x)$ of Lemma~\ref{decomp2}. It follows that:
\begin{eqnarray*}
d(x_{1/2}y_{1/2}) &=&f_d(x_{1/2})y_{1/2}+x_{1/2}f_d(y_{1/2})+2d(e)(x_{1/2}y_{1/2})_1-2d(e)(x_{1/2}y_{1/2})_0\\
\end{eqnarray*}
and we have
$$h_d((x_{1/2}y_{1/2})_1)=[f_d(x_{1/2})y_{1/2}+x_{1/2}f_d(y_{1/2})]_1\text{ et }$$
$$ g_d((x_{1/2}y_{1/2})_0)=[f_d(x_{1/2})y_{1/2}+x_{1/2}f_d(y_{1/2})]_0, \textrm{ and we have } (vii).$$

We have
\begin{eqnarray*}
  d(x_1x_0) &=& f_d((x_1x_0)_{1/2}) \\
   &=& d(x_1)x_0+x_1d(x_0) \\
   &=& [h_d(x_1)+2d(e)x_1]x_0+x_1[g_d(x_0)-2d(e)x_0]\\
   &=& h_d(x_1)x_0+x_1g_d(x_0),
\end{eqnarray*}

$$f_d((x_1x_0)_{1/2})=h_d(x_1)x_0+x_1g_d(x_0), \textrm{ and we have } (viii).$$

In a similar way
\begin{eqnarray*}
  d(x_1x_{1/2}) &=& f_d(x_1x_{1/2})+2[d(e)(x_1x_{1/2})]_0-2[d(e)(x_1x_{1/2})]_1 \\
   &=& d(x_1)x_{1/2}+x_1d(x_{1/2}) \\
   &=& [h_d(x_1)+2d(e)x_1]x_{1/2}+x_1[f_d(x_{1/2})+2[d(e)x_{1/2}]_0-2[d(e)x_{1/2}]_1]\\
   &=& h_d(x_1)x_{1/2}+x_1f_d(x_{1/2})+2[d(e)x_1]x_{1/2}+2x_1[d(e)x_{1/2}]_0-2x_1[d(e)x_{1/2}]_1,
\end{eqnarray*}

$$f_d(x_1x_{1/2})=h_d(x_1)x_{1/2}+x_1f_d(x_{1/2}), \textrm{ and we have } (ix).$$
So we have
\begin{eqnarray*}
  d(x_0x_{1/2}) &=& f_d(x_0x_{1/2}) \\
   &=& d(x_0)x_{1/2}+x_0d(x_{1/2}) \\
   &=& [g_d(x_0)-2d(e)x_0]x_{1/2}+x_0[f_d(x_{1/2})+2[d(e)x_{1/2}]_0-2[d(e)x_{1/2}]_1]\\
   &=& g_d(x_0)x_{1/2}+x_0f_d(x_{1/2})-2[d(e)x_0]x_{1/2}+x_0[2[d(e)x_{1/2}]_0\\
   & &-2[d(e)x_{1/2}]_1],
\end{eqnarray*}
$$f_d(x_0x_{1/2})=g_d(x_0)x_{1/2}+x_0f_d(x_{1/2}), \textrm{ and finally } (x).$$

Conversely, once we have identities $(i)$ to $(xi)$, setting $x=x_1+x_{1/2}+x_0$ and $y=y_1+y_{1/2}+y_0$, we show that $d(xy)=d(x)y+xd(y)$.
\end{proof}

\begin{exe}
Let $A$ be the four dimensional Lie triple $K$-algebra which multiplication table in the basis $\{e,t,u,r\}$ is given by : $e^2=e+t$, $u^2=u+r$, $et=\frac12 t$, $ur=\frac12 r$, all other product being zero.  The Peirce decomposition of $A$ relative to pseudo-idempotent $e$ gives
$A_e(1)=K(e+2t)$, $A_e(1/2)=Kt$, $A_e(0)=<u,r>$. Let $d$ be a derivation of $A$. Since $e$ and $u$ are pseudo-idempotents, we have $d(t)=d(r)=0$, $d(e)=\al t$, $d(u)=\be r$.
 The derivation algebra is two dimensional.
\end{exe}

\begin{exe}
    Let's consider the four dimensional Lie triple $K$-algebra $A$ which multiplication table in the basis $\{e,t_1,t_2,v\}$ is given by : $e^2=e+t_1$, $et_1=\frac12 t_1$, $et_2=\frac12 t_2$, $ev=v$ and $vt_1=t_2$, all other products being zero. Then we have $A_e(1)=<(e+2t), v>$, $A_e(1/2)=<t_1, t_2>$, $A_e(0)=0$. Let $d$ be a derivation of $A$. We have $d(t_1)=0.$ Set $d(e)= \al_1 t_1 +\be_1 t_2.$ Thus $\al_1 t_1 +\be_1 t_2=d(e+2t_1)=h_d(e+2t_1)+2d(e)(e+2t_1)=h_d(e+2t_1)+\al_1 t_1 +\be_1t_2.$ It follows that $h_d(e+2t_1)=0.$

Setting $d(v)=\al_2(e+2t_1)+\be_2 v+\g_2 t_1+\eta_2 t_2$, relation $0=d(v^2)=2d(v)v$ gives $\al_2=\g_2=0.$ It follows that $h_d(v)=\be_2 v.$ Furthermore, relation $f_d((v t_1)_{1/2})=h_d(t_1)+f_d(t_1)$ gives $f_d(t_2)=\be_2 t_2.$

We have $d(e)=\al_1t_1+\be_1 t_2, h_d(e+2t_1)=0, h_d(v)=\be_2 v, f_d(t_1)=0, f_d(t_2)=\be_2 t_2.$ The derivation algebra is three dimensional.
\end{exe}

\begin{prop}
Let's consider a pseudo-idempotent $e\ne0$.
Subspace $J_e(1/2)=\{x_{1/2}\in A_e(1/2) \mid x_{1/2}A_e(1/2)=0\}$ is a caracteristic ideal and $A/J_e(1/2)$ is a Lie triple algebra with $\overline{e}$ as idempotent.
\end{prop}

\begin{proof} Let $x_{1/2}\in J_e(1/2)$, $a_{1/2} \in A_e(1/2)$ and $y_\la\in A_e(\la)$ ($\la=0,1$). We have $[(x_{1/2}y_\la)a_{1/2}]_\la=[(a_{1/2}y_\la)x_{1/2}]_\la=0$ and $[(x_{1/2}y_\la)a_{1/2}]_{1-\la}=[(a_{1/2}y_\la)x_{1/2}]_{1-\la}=0$, and then $(x_{1/2}y_\la)a_{1/2}=0$. Hence $A_e(\la) J_e(1/2)\subseteq J_e(1/2)$, and it follows that $A J_e(1/2)\subseteq J_e(1/2)$. $J_e(1/2)$ is an ideal of $A$. Since $t\in L_{1/2}\subseteq J_e(1/2)$, $\overline{e}$ is an idempotent of quotient algebra $A/J_e(1/2)$.
	
Let's consider now $d\in Der_K(A)$, $x_{1/2}\in J_e(1/2)$ and $a_{1/2} \in A_e(1/2)$.
We have $0=d(x_{1/2}a_{1/2})=x_{1/2}d(a_{1/2})+d(x_{1/2})a_{1/2}$. But $d(x_{1/2})=f_d(x_{1/2})\in A_{1/2}$ and $d(a_{1/2})=f_d(a_{1/2})+2(d(e)a_{1/2})_0-2(d(e)a_{1/2})_1,$ it follows that  $x_{1/2} d(a_{1/2})=0$ because $x_{1/2}(d(e)a_{1/2})_1=x_{1/2}(d(e)a_{1/2})_0$. So $d(x_{1/2})a_{1/2}=0$, and $d(x_{1/2})\in J_e(1/2)$. We conclude that $d(J_e(1/2))\subseteq J_e(1/2)$.
\end{proof}

\section{Dimensionally nilpotent Lie algebras}

\begin{defi}
Let $A$ be a $n+1$ finite dimensional $K$-algebra . If there is a nilpotent $K$-derivation $d$ of $A$ such that $d^{n+1}=0$ and $d^{n}\ne 0$, $d$ is said to be dimensionally nilpotent, and so is the algebra $A$, though $A$ is not necessarily nilpotent.

If so it is, there is a basis $\{e_0, e_1,\ldots, e_n\}$ of $A$ such that $d(e_i)= e_{i+1}$
$(i=0,\ldots n-1)$ and $d(e_n)= 0$ and the basis $\{e_0, e_1,\ldots, e_n\}$ is said to be \textit{adapted} to $d$.
\end{defi}

\begin{exe}\cite[Exemple~2.5]{Micali}\label{gametique}
Let $K$ be a commutative field of characteristic $\ne2$
and $A = G(n + 1,2)$ the gametic diploid algebra  with $n+ 1$
alleles. Its  multiplication table in the natural basis $\{a_0,\ldots,a_n\}$ is given by $a_ia_j=\frac12 a_i+\frac12 a_j$. We know that the mapping $\w:A\rightarrow K, a_i\mapsto 1$ is a weight function and if we set $e_i=a_0-a_i$ $(i\ne0)$ then $\{e_1,\ldots, e_n\}$ is a basis of the ideal $N = \ker\w$ and $e_0=a_0$ is an idempotent of $A$ such that $\{e_0, e_1,\ldots, e_n\}$ is the canonical basis of $A$, so $e_0e_i =\frac12 e_i$ $(i = l,\ldots,n)$ and if $d$ is a derivation of $A$, $e_0d(e_i) = \frac12 d(e_i)$ $(i = 1,\ldots,n)$, because $d(e_0) \in N$ and $N$ is a zero algebra. So we just need to define $d : A \rightarrow A$ by $d(e_i) = e_{i+1}$ $(i = 1,\ldots,n-1)$, $d(e_0) = e_1$ and $d(e_n) = 0$. It follows that $d^{n+1} = 0$ and $d^n\ne 0$, showing the gametic algebra $A =G(n + 1,2)$ is dimensionally nilpotent.
\end{exe}

\begin{exe}
	Let $K$ be a commutative field of characteristic $\ne 2$ and $A$ the $n+1$ dimensional commutative $K$-algebra, which multiplication table in the basis $\{e_0,e_1,\ldots,e_n\}$ is given by $e_0e_i=\frac12 e_i$ $(i=1,\ldots,n)$, $e_0^2=e_0+e_n$, all other product being zero. If $d$ is a derivation of $A$, $e_0d(e_i)=\frac12 d(e_i)$ $(i=1,\ldots,n)$ because $d(e_0)\in N=<e_1,\ldots,e_n>$ and $N$ is a zero algebra. Here, we just need again to define $d:A\rightarrow A$ by $d(e_i)=e_{i+1}$ $(i=1,\ldots,n-1)$, $d(e_0) = e_1$ and $d(e_n) = 0$. We have $d^{n+1} = 0$ and $d^n\ne 0$, that shows the algebra $A$ is dimensionally nilpotent. Since $Ke_n$ is an ideal, the quotient algebra $A/Ke_n$ is isomorphic to $G(n,2)$.
\end{exe}

\subsection{Basic tools}

\begin{theo}[\cite{Osborn}]\label{osb}
Let $K$ be a perfect field of characteristic $\ne2$ and $3$ and $A$ finite dimensional $K$-Jordan algebra, dimensionally nilpotent. Then either $A$ is nilpotent or $\dim_K(A/rad(A))= 1$.
\end{theo}

\begin{rem}
Let $A$ be a dimensionally nilpotent Lie triple non nilalgebra. Because of Theorem~\ref{idemp} we consider two cases :
\begin{itemize}
	\item[$1)$] $A$ has an idempotent $e$. Since the ideal $J$ is caracteristic, the quotient algebra $\overline{A}=A/J$ is a dimensionally nilpotent Jordan algebra. Because of Theorem~\ref{osb} we have $\dim_K(\overline{A}/rad(\overline{A})=1$ and since $rad(\overline{A})\simeq rad(A)/J$, the first isomorphism theorem gives $A/rad(A)\simeq \overline{A}/rad(\overline{A})$ and $\dim_KA/rad(A)=1$. Then we can write $A=Ke\oplus N$, with $N=rad(A)$.
	\item[$2)$] $A$ has a pseudo-idempotent $e$. Since the ideal $J_e(1/2)$ is caracteristic, the quotient algebra $\overline{A}=A/J_e(1/2)$ is a dimensionally Lie triple algebra with $\overline{e}$ as idempotent. Because of $1)$ we can write $\overline{A}=K\overline{e}\oplus \overline{N}$ with $\overline{N}=rad(\overline{A})$. So we have $A=Ke\oplus N$ with $N=rad(A)$.
\end{itemize}
\end{rem}

\begin{lem}\label{lem1}
	Let $x$, $y \in N$ such that $x\neq 0$ and $\al \in K$. If $xy=\al y$ then $\al=0$ or $y=0$.
\end{lem}

\begin{proof}
	Since $N$ is nilpotent, there is $m\in \N^{*}$ such that $L_x^m(y)=\al^m y=0$, $L_x$ being the multiplicative operator by $x$. Then $\al=0$ or $y=0$.
\end{proof}

From now on, throughout the paper, $A$ is a dimensionally nilpotent Lie triple non nilalgebra of dimension $n+1$, with  $\{e_0, e_1,\ldots, e_n\}$ an adapted basis to the derivation $d$. We can consider $e_0$ either, as an idempotent, or a pseudo-idempotent. In the last case, $e_0^2=e_0+t$, $e_0t=\frac{1}{2}t$ and $t^2=0$ implies $d(t)=0$ (Lemma~\ref{d(t)}), that means $t=\al e_n$ with $\al\in K$.
Since $t\in A_e(1/2)$, if $\al\ne0$, then $e_n\in A_e(1/2)$.

\begin{lem}\label{lem2}
We have:

$(i)$ $e_0e_n=\la_n e_n$

$(ii)$ $e_ke_n=0$ with $1\leq k \leq n$
\end{lem}

\begin{proof}
Let's write $e_0e_n=\sum^n_{i=0} \la_i e_i$. Deriving $k$ times successively, we have $e_ke_n=\sum^{n-k}_{i=0} \la_i e_{i+k}$. With $k=n$, it follows that $e_n^2=\la_0 e_n$ and because of Lemma~\ref{lem1} we have $\la_0=0$. Set $k=n-1$, one has $e_{n-1}e_n=\la_1e_n$. That implies $\la_1=0$. And so on, we have $\la_0=\la_1=\cdots=\la_{n-1}=0$, $e_0e_n=\la_ne_n$. Deriving successively $e_0e_n$ it follows that $e_ke_n=0$ with $1\leq k \leq n$.
\end{proof}

\begin{lem}\label{lem3}
We have :
\begin{itemize}
	\item[$(i)$] $e_0e_k=\la_ke_k+\sum^n_{i=k+1} a_{k,i} e_i$ with $1\leq k \leq n-1$;
\item[$(ii)$] $e_ie_k=\sum^n_{j=k+1} \g_{ikj} e_j$ with $1\leq i\leq  k\leq  n-1$;
\item[$(iii)$] $\la_k \in \{0, \frac{1}{2}, 1\}$ with $1\leq k \leq n$.
\end{itemize}
\end{lem}

\begin{proof}
Reason by recurrence on $n$. With $n=1$ the multiplication table of the algebra $A$ is given by $e_0^2=e_0$, $e_0e_1=\frac12 e_1$, $e_1^2=0$  and the lemma is satisfied. Assume the lemma is true until an order $n$. Because of Lemma~\ref{lem2} the subspace $I_{n+1}=Ke_{n+1}$ is a $d$-invariant ideal of $A$. The quotient algebra $A/I_{n+1}$ is dimensionally nilpotent of dimension $n+1$. By the hypothesis, we have $\overline{e}_0\overline{e}_k=\la_k \overline{e}_k+\sum_{i=k+1}^na_{k,i}\overline{e}_i$ and $\overline{e}_i\overline{e}_k=\sum_{j=k+1}^n\g_{ikj}\overline{e}_j$, with $1\leq i\leq k\leq n$. Otherwise $e_0e_k=\la_k e_k+\sum_{i=k+1}^na_{k,i}e_i+a_{k,n+1}e_{n+1}$ and $e_ie_k=\sum_{j=k+1}^n \g_{ikj}e_j+\g_{ik,n+1}e_{n+1}$;
 and results $(i)$ and $(ii)$ follow.

Now we just need to show $(iii)$. Since $2L_{e_0}^3-3L_{e_0}^2+L_{e_0}=0$, with $L_{e_0}$ being the multiplicative operator by $e_0$, applying it to $e_k$ we have $2\la_k^3-3\la_k^2+\la_k=0$, soit $\la_k \in \{0, \frac{1}{2}, 1\}$.
\end{proof}


\subsection{Example of low dimensions}\label{Sect41}
Here we deal with cases $1\leq n \leq 4$. Let $A$ be a dimensionally nilpotent Lie triple algebra, of dimension $n+1$ and $\{e_0, e_1,\ldots, e_n\}$ be a basis adapted to $d$. We have
$\ker d= Ke_n$. Since $e_0$ is an idempotent or a pseudo-idempotent, $e_1=d(e_0)\in A_e(1/2)$, i.e $e_0e_1=\frac{1}{2}e_1$. Deriving this we have
$e_0e_2+e_1^2=\frac{1}{2}e_2$, that means
\begin{eqnarray}\label{Eq2}
\la_2+\g_{112}=\frac{1}{2} \textrm{ et } a_{2,k}+\g_{11k}=0 \quad (3\leq k\leq n).
\end{eqnarray}
We also have $e_1^2=\sum_{k=2}^{n}\g_{11k}e_k$ which derivative is $2 e_1e_2=\sum_{k=2}^{n-1}\g_{11k}e_{k+1}=\sum_{k=3}^{n}\g_{11,k-1}e_{k}$, that means
$2\g_{12k}=\g_{11,k-1}$ $(3\leq k \leq n)$. Let's derive for the second time $e_0e_1=\frac{1}{2}e_1$. We have $e_0e_3+3e_1e_2=\frac{1}{2}e_3$, that means
\begin{eqnarray}\label{Eq3}
\la_3+3\g_{123}=\frac{1}{2} \textrm{ et } a_{3,k}+3\g_{1,2,k}=0\quad (4\leq k\leq n).
\end{eqnarray}
However we have, $d(e_0e_2)=e_0e_3+e_1e_2=\la_2e_3+\sum_4^n a_{2,k-1}e_k$, that implies

\begin{eqnarray}\label{Eq4}
\la_3+\g_{123}=\la_2 \textrm{ et } a_{3,k}+\g_{12k}=a_{2,k-1}\quad (4\leq k\leq n).
\end{eqnarray}

So $(\ref{Eq4})$ implies $\g_{123}\in \{-1,-\frac{1}{2},0,\frac{1}{2},1\}$. So it is necessary to take $\la_3=\frac{1}{2}$ in $(\ref{Eq3})$. Whence $\la_3=\la_2=\frac{1}{2}$ and $\g_{123}=\g_{112}=0$ if $3\leq n$.

Deriving $e_0e_3+3e_1e_2=\frac{1}{2}e_3$, one has $e_0e_4+4e_1e_3+3e_2^2=\frac{1}{2}e_4$, that means

\begin{eqnarray}\label{Eq5}
\la_{4}+4\g_{134}+3\g_{224}=\frac{1}{2}
\end{eqnarray}
However, $d(e_0e_3)=e_0e_4+e_1e_3=\la_3e_4+\sum_{k=5}^na_{3,k-1}e_k$, which implies
\begin{eqnarray}\label{Eq6}
\la_4+\g_{134}=\la_3 \textrm{ et } a_{4,k}+\g_{13k}=a_{3,k-1}\quad (5\leq k\leq n)
\end{eqnarray}
We also have $e_1e_2=\sum_{k=4}^n\g_{12k}e_k$ because $\g_{123}=0$ and $d(e_1e_2)=e_1e_3+e_2^2=\sum_{k=5}^n\g_{12,k-1}e_k$, which implies $\g_{223}=0$ and $\g_{134}+\g_{224}=0$.

\vspace{0,5cm}

\noindent\textbf{Case $\dim_KA=2$ i.e $n=1$.}

We obviously have $e_0e_1=\frac{1}{2}e_1$, $e_1^2=0$, $e_0^2=e_0$ or $e_0^2=e_0+e_1$ all other product being zero.

\medskip
\begin{center}
	\begin{tabular}{|l|l|l|l|}
  \hline
    & $e_0$ & $e_1$ \\\hline
  $e_0$ & $e_0$&$\frac{1}{2}e_1$ \\\hline
  $e_1$& &$0$ \\\hline
\end{tabular}
\hspace{1cm}\hspace{1cm}
\begin{tabular}{|l|l|l|l|}
  \hline
    & $e_0$ & $e_1$ \\\hline
  $e_0$ & $e_0+e_1$&$\frac{1}{2}e_1$ \\\hline
  $e_1$& &$0$ \\\hline
\end{tabular}
\end{center}

\medskip
\noindent\textbf{Case $\dim_KA=3$ i.e $n=2$.}

Because of $(\ref{Eq2})$ we have $\la_2+\g_{112}=\frac{1}{2}$. Let's discuss the possible values of $\la_2$.

$\ast$ $\la_2=0$ $\Rightarrow$ $\g_{112}=\frac{1}{2}$, so
$e_0^2=e_0$, $e_0e_1=\frac{1}{2}e_1$, $e_1^2=\frac{1}{2}e_2$ all other product being zero.

\medskip
\begin{center}
\begin{tabular}{|l|l|l|l|}
  \hline
    & $e_0$ & $e_1$&$e_2$ \\\hline
  $e_0$ & $e_0$&$\frac{1}{2}e_1$&$0$ \\\hline
  $e_1$& &$\frac{1}{2}e_2$& $0$\\\hline
  $e_2$& && $0$\\\hline
\end{tabular}
\end{center}

\medskip

$\ast$ $\la_2=\frac{1}{2}$ $\Rightarrow$ $\g_{112}=0$, so
$e_0^2=e_0$ or $e_0^2=e_0+e_2$, $e_0e_1=\frac{1}{2}e_1$, $e_0e_2=\frac{1}{2}e_2$ all other product being zero.

\medskip
\begin{center}
\begin{tabular}{|l|l|l|l|}
  \hline
    & $e_0$ & $e_1$&$e_2$ \\\hline
  $e_0$ & $e_0$&$\frac{1}{2}e_1$&$\frac{1}{2}e_2$ \\\hline
  $e_1$& &$0$& $0$\\\hline
  $e_2$& && $0$\\\hline
\end{tabular}
\hspace{1cm}
\begin{tabular}{|l|l|l|l|}
  \hline
    & $e_0$ & $e_1$&$e_2$ \\\hline
  $e_0$ & $e_0+e_2$&$\frac{1}{2}e_1$&$\frac{1}{2}e_2$ \\\hline
  $e_1$& &$0$& $0$\\\hline
  $e_2$& && $0$\\\hline
\end{tabular}
\end{center}

\medskip

$\ast$ $\la_2=1$ $\Rightarrow$ $\g_{112}=-\frac{1}{2}$, so
$e_0^2=e_0$, $e_0e_1=\frac{1}{2}e_1$, $e_0e_2=e_2$, $e_1^2=-\frac{1}{2}e_2$ all other product being zero.

\vspace{0,5cm}
\begin{center}
\begin{tabular}{|l|l|l|l|}
  \hline
    & $e_0$ & $e_1$&$e_2$ \\\hline
  $e_0$ & $e_0$&$\frac{1}{2}e_1$&$e_2$ \\\hline
  $e_1$& &$-\frac{1}{2}e_2$& $0$\\\hline
  $e_2$& && $0$\\\hline
\end{tabular}
\end{center}

\medskip

\noindent\textbf{Case $\dim_KA=4$ i.e $n=3$.}

Because of the preliminary calculations, $\la_3=\la_2=\la_1=\frac{1}{2}$, $\g_{112}=\g_{123}=0$ and $a_{2,3}+\g_{113}=0$. So  $e_0e_3=\frac{1}{2}e_3$ $\Rightarrow$ $e_3\in A_{\frac{1}{2}}$
and $e_1^2=\g_{113}e_3$. Since $A_{\frac{1}{2}}^2\subseteq A_{0}+A_{1}$ we have $\g_{113}=0$ implying $a_{2,3}=0$ and finally $e_0e_2=\frac{1}{2}e_2$. So we have the following multiplication table : $e_0e_1=\frac{1}{2}e_1$, $e_0e_2=\frac{1}{2}e_2$, $e_0e_3=\frac{1}{2}e_3$, $e_0^2=e_0$ or $e_0^2=e_0+e_3$, all other product being zero.

\vspace{0,5cm}
\begin{center}
\begin{tabular}{|l|l|l|l|l|}
  \hline
    & $e_0$ & $e_1$&$e_2$ &$e_3$\\\hline
  $e_0$ & $e_0$&$\frac{1}{2}e_1$&$\frac{1}{2}e_2$ &$\frac{1}{2}e_3$\\\hline
  $e_1$& &$0$& $0$&$0$\\\hline
  $e_2$& && $0$&$0$\\\hline
  $e_3$& && &$0$\\\hline
\end{tabular}
 \hspace{1cm}
 \begin{tabular}{|l|l|l|l|l|}
  \hline
    & $e_0$ & $e_1$&$e_2$ &$e_3$\\\hline
  $e_0$ & $e_0+e_3$&$\frac{1}{2}e_1$&$\frac{1}{2}e_2$ &$\frac{1}{2}e_3$\\\hline
  $e_1$& &$0$& $0$&$0$\\\hline
  $e_2$& && $0$&$0$\\\hline
  $e_3$& && &$0$\\\hline
\end{tabular}
\end{center}

\medskip

\noindent\textbf{Case $\dim_KA =5$ i.e $n=4$.}

\vspace{0,5cm}

$\ast$ $\la_4=0$ $\Rightarrow$ $\g_{134}=\frac{1}{2}$ because of $(\ref{Eq6})$.
Since $e_1e_3=\g_{134}e_4=\frac{1}{2}e_4$ we have $e_3\in A_{\frac{1}{2}}$ because  $A_{\frac{1}{2}}^2\subseteq A_{0}+A_{1}$, $A_{\frac{1}{2}}A_{1}\subseteq A_{\frac{1}{2}}$, $A_{\frac{1}{2}}A_{0}\subseteq A_{\frac{1}{2}}$. So $e_0e_3=\frac{1}{2}e_3$ and  $a_{3,4}=0$ $\Rightarrow$ $\g_{124}=0$ because of $(\ref{Eq4})$ and finally $\g_{113}=a_{2,3}=0$.
In the same way $e_2^2=-\frac{1}{2}e_4$ $\Rightarrow$ $e_2\in A_e(0)$ or $e_2\in A_{\frac{1}{2}}$, because $A_e(1/2)^2\subseteq A_e(0)+A_e(1)$ and $A_e(0)^2\subseteq A_e(0)$. But $e_0e_2=\frac{1}{2}e_2+a_{2,4}e_4$ $\Rightarrow$ $e_2\in A_{\frac{1}{2}}$ so $a_{2,4}=0=\g_{114}$ because of $(\ref{Eq2})$. Whence the following multiplication table

\vspace{0,5cm}
\begin{center}
\begin{tabular}{|l|l|l|l|l|l|}
  \hline
    & $e_0$ & $e_1$&$e_2$ &$e_3$&$e_4$\\\hline
  $e_0$ & $e_0$&$\frac{1}{2}e_1$&$\frac{1}{2}e_2$ &$\frac{1}{2}e_3$&$0$\\\hline
  $e_1$& &$0$& $0$&$\frac{1}{2}e_4$&$0$\\\hline
  $e_2$& & &$-\frac{1}{2}e_4$&$0$&$0$\\\hline
  $e_3$& && &$0$&$0$\\\hline
  $e_4$& && &&$0$\\\hline
\end{tabular}
\end{center}

\vspace{0,5cm}

$\ast$ $\la_4=\frac{1}{2}$ $\Rightarrow$ $\g_{134}=\g_{224}=0$ because of $(\ref{Eq5})$ and $(\ref{Eq6})$. We have $e_1e_2=\g_{1,2,4}e_4\in A_{\frac{1}{2}}$ $\Rightarrow$ $e_2\in A_e(0)$ or $e_2\in A_e(1)$ this is a contradiction (because $e_0e_2=\frac{1}{2}e_2+a_{2,3}e_3$)
so $\g_{124}=0$ and then $a_{3,4}=\g_{113}=0$. Whence the following multiplication table:
\vspace{0,5cm}

\begin{center}
\begin{tabular}{|l|l|l|l|l|l|}
  \hline
    & $e_0$ & $e_1$&$e_2$ &$e_3$&$e_4$\\\hline
  $e_0$ & $e_0$&$\frac{1}{2}e_1$&$\frac{1}{2}e_2$ &$\frac{1}{2}e_3$&$\frac{1}{2}e_4$\\\hline
  $e_1$& &$0$& $0$& $0$&$0$\\\hline
  $e_2$& & &$0$&$0$&$0$\\\hline
  $e_3$& && &$0$&$0$\\\hline
  $e_4$& && &&$0$\\\hline
\end{tabular} \hspace{0,5cm}
\begin{tabular}{|l|l|l|l|l|l|}
  \hline
    & $e_0$ & $e_1$&$e_2$ &$e_3$&$e_4$\\\hline
  $e_0$ & $e_0+e_4$&$\frac{1}{2}e_1$&$\frac{1}{2}e_2$ &$\frac{1}{2}e_3$&$\frac{1}{2}e_4$\\\hline
  $e_1$& &$0$& $0$& $0$&$0$\\\hline
  $e_2$& & &$0$&$0$&$0$\\\hline
  $e_3$& && &$0$&$0$\\\hline
  $e_4$& && &&$0$\\\hline
\end{tabular}
\end{center}
\vspace{0,5cm}

$\ast$ $\la_4=1$ $\Rightarrow$ $e_4\in A_e(1)$ $\Rightarrow$ $\g_{134}=-\frac{1}{2}$

We have $e_1e_3=-\frac{1}{2}e_4$ $\Rightarrow$ $e_3\in A_{\frac{1}{2}}$ $\Rightarrow$ $e_0e_3=\frac{1}{2}e_3$ $\Rightarrow$ $a_{3,4}=0$. So $\g_{124}=\g_{113}=a_{2,3}=0$.
In the same way $e_2^2=\frac{1}{2}e_4$ $\Rightarrow$ $e_2\in A_{\frac{1}{2}}$ (because $e_2$ can not be in $A_e(1)$), $e_0e_2=\frac{1}{2}e_2$ $\Rightarrow$ $a_{2,4}=\g_{114}=0$.
Whence this table
\vspace{0,5cm}

\begin{center}
\begin{tabular}{|l|l|l|l|l|l|}
  \hline
    & $e_0$ & $e_1$&$e_2$ &$e_3$&$e_4$\\\hline
  $e_0$ & $e_0$&$\frac{1}{2}e_1$&$\frac{1}{2}e_2$ &$\frac{1}{2}e_3$&$e_4$\\\hline
  $e_1$& &$0$& $0$&$-\frac{1}{2}e_4$&$0$\\\hline
  $e_2$& & &$\frac{1}{2}e_4$&$0$&$0$\\\hline
  $e_3$& && &$0$&$0$\\\hline
  $e_4$& && &&$0$\\\hline
\end{tabular}
\end{center}

\subsection{Main results in general case}

\begin{theo}[Main theorem]\label{2TP}
Let $A$ be dimensionally nilpotent Lie triple non nilalgebra. Let $\{e_0,e_1, \ldots, e_n\}$ be an adapted basis of $A$. Then:

\noindent$1°)$ If $n=2p+1$, the multiplication table of $A$ is one of the two following:
\begin{itemize}
	\item[$(i)$] $e_0^2=e_0$, $e_0e_i=\frac{1}{2}e_i$ $(1\leq i \leq 2p+1)$, all other product being zero.
	\item[$(i')$] $e_0^2=e_0+e_{2p+1}$, $e_0e_i=\frac{1}{2}e_i$ $(1\leq i \leq 2p+1)$, all other product being zero.
\end{itemize}

\noindent$2°)$ If $n=2p$, the multiplication table of $A$ is one of the four following :
\begin{itemize}
	\item[$(i)$] $e_0^2=e_0$, $e_0e_i=\frac{1}{2}e_i$ $(1\leq i \leq 2p)$, all other product being zero.
	\item[$(i')$] $e_0^2=e_0+e_{2p}$, $e_0e_i=\frac{1}{2}e_i$ $(1\leq i \leq 2p)$, all other product being zero.
	\item[$(ii)$] $e_0^2=e_0$,  $e_0e_i=\frac{1}{2}e_i$ $(1\leq i \leq 2p-1)$, $e_0e_{2p}=0$, $e_ie_{2p-i}=\frac{1}{2}(-1)^{i-1}e_{2p}$, $(1\leq i \leq p)$, all other product being zero.
\item[$(iii)$] $e_0^2=e_0$, $e_0e_i=\frac{1}{2}e_i$ $(1\leq i \leq 2p-1)$, $e_0e_{2p}=e_{2p}$,
$e_ie_{2p-i}=\frac{1}{2}(-1)^{i}e_{2p}$ $(1\leq i \leq p)$, all other product being zero.
\end{itemize}
\end{theo}

\begin{proof}
Reason by recurrence on $n$. Subsection~\ref{Sect41} shows that the theorem is true when $n\leq 4$. Assume it is true until an order $n>4$ and let's show it remains true for $n+1$. Integer $n$ being either even or odd, we consider two cases :

$\mathbf{1)}$ $n=2p$ is even. The multiplication table of $A$ has the following form $e_0e_k=\frac12 e_k + 
a_{k,2p+1}e_{2p+1}$ $(1\leq k\leq 2p-1)$, $e_0e_{2p}=\la_{2p}e_{2p}+a_{2p,2p+1}e_{2p+1}$, $e_0e_{2p+1}=\la_{2p+1}e_{2p+1}$ and $e_ie_{2p-i}=\varepsilon_ie_{2p}+\g_{i,2p-i,2p+1}e_{2p+1}$, with $\varepsilon_i=0$, $\varepsilon_i=\frac{1}{2}(-1)^{i-1}$ or $\varepsilon_i=\frac{1}{2}(-1)^{i}$ $(i=1,\ldots,p)$ according to $\la_{2p}=\frac12$, $\la_{2p}=0$ or $\la_{2p}=1$, respectively. We have $d(e_ie_{2p-i})=e_ie_{2p+1-i}+e_{i+1}e_{2p-i}=\varepsilon_i e_{2p+1}$, and then the following system
	
\begin{displaymath}
\begin{cases}
e_1e_{2p}+e_{2}e_{2p-1}=\varepsilon_1 e_{2p+1}, & \hbox{} \\
\cdots\cdots, & \\
e_ie_{2p+1-i}+e_{i+1}e_{2p-i}=\varepsilon_i e_{2p+1}, & \hbox{}\\
\cdots\cdots, & \\
e_pe_{p+1}+e_{p+1}e_{p}=2e_pe_{p+1}=\varepsilon_p e_{2p+1}.
\end{cases}\tag{S}
\end{displaymath}
We see that $e_pe_{p+1}=\frac{1}{2}\varepsilon_p e_{2p+1}$, $e_{p-1}e_{p+2}=(\varepsilon_{p-1}-\frac{1}{2}\varepsilon_p)e_{2p+1}=\frac{3}{2}\varepsilon_{p-1}e_{2p+1}$, $e_{p-i}e_{p+1+i}=\frac{2i+1}{2}\varepsilon_{p-i}e_{2p+1}$, $e_1e_{2p}=\frac{2p-1}{2}\varepsilon_1 e_{2p+1}$. But $d(e_0e_{2p})=e_0e_{2p+1}+e_1e_{2p}=\la_{2p}e_{2p+1}$, that means $e_1e_{2p}=(\la_{2p}-\la_{2p+1})e_{2p+1}$, and also $\la_{2p}-\la_{2p+1}=\frac{2p-1}{2}\varepsilon_1$.

\begin{itemize}
\item[If] $\la_{2p}=0$, we have $\varepsilon_1=\frac{1}{2}$ and $\la_{2p+1}=-\frac{2p-1}{4}\not\in \{0,\frac{1}{2},1\}$, impossible.
\item[If] $\la_{2p}=1$, we have $\varepsilon_1=-\frac{1}{2}$ and $\la_{2p+1}=1+\frac{2p-1}{4}=\frac{2p+3}{4}\not\in \{0,\frac{1}{2},1\}$, impossible.
\item[Hence]   $\la_{2p}=\frac{1}{2}$, $\varepsilon_1=0$ and $\la_{2p+1}=\frac12$.
\end{itemize}
Since all the $\la_k$ are equal to $\frac12$ ($k\ne 2p+1$), applying $2L_{e_0}^3-3L_{e_0}^2+L_{e_0}=0$ to $e_k$, it follows that $2a_{k,2p+1}\la_{2p+1}(\la_{2p+1}-1)=0$. If $\la_{2p+1}=\frac12$, then we have $a_{k,2p+1}=0$, which means $e_0e_k=\frac12 e_k$ for all $k$. Hence $e_ie_j=0$, with $1\leq i\leq j\leq n$.
\medskip

$\mathbf{2)}$ $n=2p-1$ is odd. The multiplication table of $A$ has the following form $e_0e_k=\frac12 e_k + a_{k,2p}e_{2p}$ ($k=1,\ldots,2p-1$), $e_0e_{2p}=\la_{2p}e_{2p}$ and $e_ie_{2p-1-i}=\g_{i,2p-1-i,2p}e_{2p}$ ($i=1,\ldots,p-1$).
Deriving this last relation, we have $e_ie_{2p-i}+e_{i+1}e_{2p-1-i}=0$, and the following system
\begin{displaymath}
\begin{cases}
e_1e_{2p-1}+e_{2}e_{2p-2}=0, & \hbox{} \\
e_2e_{2p-2}+e_{3}e_{2p-3}=0, & \hbox{} \\
\cdots\cdots, & \\
e_{p-1}e_{p+1}+e_{p}^2=0.
\end{cases}\tag{$S'_p$}
\end{displaymath}
So $e_1e_{2p-1}=-e_2e_{2p-2}=e_3e_{2p-3}=\cdots =(-1)^{i-1}e_ie_{2p-i}=\cdots=(-1)^{p-1}e^2_p$, that means $e_ie_{2p-i}=(-1)^{i-1}e_1e_{2p-1}$. However, since $d(e_0e_{2p-1})=e_0e_{2p}+e_1e_{2p-1}=\frac{1}{2}e_{2p}$, we have $e_1e_{2p-1}=(\frac{1}{2}-\la_{2p})e_{2p}$. Because $\la_{2p}\in\{0,\frac{1}{2},1\}$, we consider three situations :

\begin{itemize}
	\item[If] $\la_{2p}=0$, we have $e_1e_{2p-1}=\frac{1}{2}e_{2p}$,  $e_ie_{2p-i}=\frac{1}{2}(-1)^{i-1}e_{2p}$ $(i=1,\ldots, p)$ and $e_0^2=e_0$.
	\item[If] $\la_{2p}=1$, we have $e_1e_{2p-1}=-\frac{1}{2}e_{2p}$,  $e_ie_{2p-i}=\frac{1}{2}(-1)^{i}e_{2p}$ $(i=1,\ldots, p)$ and $e_0^2=e_0$.
	\item[If] $\la_{2p}=\frac{1}{2}$, we have $e_1e_{2p-1}=e_ie_{2p-i}=0$ $(i=1\ldots p)$ and $2a_{i,2p}\la_{2p}(\la_{2p}-1)=0$ shows that $a_{i,2p-i}=0$. Hence $e_0e_i=\frac{1}{2}e_i$ $(i=1,\ldots,p)$. Furthermore we have, either $e_0^2=e_0$, or $e_0^2=e_0+e_{2p}$.
\end{itemize}
For cases $\la_{2p}=0$ and $\la_{2p}=1$, We just need to show $e_ie_j=0$ for $i+j<2p$. The following lemma completes the proof of the theorem. And Note~\ref{NOTE} shows that all algebras defined in this theorem are Lie triple.
\end{proof}

\begin{lem}
$e_0e_i=\frac12 e_i$ for $i=1,\ldots,n-1$.
\end{lem}
\begin{proof}
One has $e_ie_{2k-i}=\g_{i,2k-i,n}e_{n}$ for $i=1,\ldots,k-1$. Deriving this we have $e_ie_{2k-i+1}+e_{i+1}e_{2k-i}=0$. By varying $i$ we have the following system
\begin{displaymath}
\begin{cases}
e_1e_{2k}+e_{2}e_{2k-1}=0, & \hbox{} \\
e_2e_{2k-1}+e_{3}e_{2k-2}=0, & \hbox{} \\
\cdots\cdots, & \\
e_{k-1}e_{k+2}+e_{k}e_{k+1}=0, & \hbox{}\\
e_ke_{k+1}+e_{k+1}e_{k}=2e_ke_{k+1}=0.
\end{cases}\tag{$S_k$}
\end{displaymath}
Going up the lines of this system we see that $e_ie_{2k+1-i}=0$ for $i=1,\ldots,k$ in particular $e_1e_{2k}=0$, so $e_0e_{2k+1}+e_1e_{2k}=\frac12 e_{2k+1}$ and $e_0e_{2k+1}=\frac12 e_{2k+1}$, $k=0,\ldots,p-1$.

Now we make a recurrence on $n$. Assume it is true until an order $n$. We distinguish two cases :
\begin{itemize}
	\item $n+1=2p+1$ is odd. We have $e_0e_{n+1}=\frac12 e_{n+1}$, which imposes $e_0e_n=\frac12 e_n$.
	\item $n+1=2p$ is even. Since $n=2p-1$ is odd, we have $e_0e_{n}=\frac12 e_n$.
\end{itemize}
\end{proof}

Since every commutative Jordan algebra is a Lie triple algebra, we have the following result:
\begin{coro}\label{Jordan}
Let $A$ be a commutative Jordan non nilalgebra, dimensionally nilpotent. Let $\{e_0,e_1, \ldots, e_n\}$ be an adapted basis of $A$. Then:

\noindent$1°)$ If $n=2p+1$, the multiplication table of $A$ is:

$e_0^2=e_0$, $e_0e_i=\frac{1}{2}e_i$ $(1\leq i \leq 2p+1)$, all other product being zero.

\noindent$2°)$ If $n=2p$, the multiplication table of $A$ is one of the following three tables :
\begin{itemize}
	\item[$(i)$] $e_0^2=e_0$, $e_0e_i=\frac{1}{2}e_i$ $(1\leq i \leq 2p)$, all other product being zero.
	\item[$(ii)$] $e_0^2=e_0$,  $e_0e_i=\frac{1}{2}e_i$ $(1\leq i \leq 2p-1)$, $e_0e_{2p}=0$, $e_ie_{2p-i}=\frac{1}{2}(-1)^{i-1}e_{2p}$, $(1\leq i \leq p)$, all other product being zero.
	\item[$(iii)$] $e_0^2=e_0$, $e_0e_i=\frac{1}{2}e_i$ $(1\leq i \leq 2p-1)$, $e_0e_{2p}=e_{2p}$,
	$e_ie_{2p-i}=\frac{1}{2}(-1)^{i}e_{2p}$ $(1\leq i \leq p)$, all other product being zero.
\end{itemize}
\end{coro}

\begin{proof}
	

The Corollary follows from Theorem~\ref{2TP} knowing that Jordan algebras do not admit pseudo-idempotent.
\end{proof}

\begin{note}\label{NOTE}
	1) Multiplication tables in Theorem~\ref{2TP}~$1)~(i)$ and in Corollary~\ref{Jordan}~$1)$ when $n=2p+1$ is odd are those of \textit{gametic algebras} $G(2p+2,2)$ (Example~\ref{gametique}). It is the same for those defined in Theorem~\ref{2TP}~$2)~(i)$ and in Corollary~\ref{Jordan}~$2)~(i)$ when $n=2p$ is even. These are gametic algebras $G(2p+1,2)$. They are caracterized as elementary train algebras with equation $x^2-\w(x)x=0$, in which $\w:A\rightarrow K$, $e_0\mapsto 1$, $e_i\mapsto 0$ is a homomorphism of algebras.
	
	2) Multiplication tables in Theorem~\ref{2TP}~$2)~(ii)$ and in Corollary~\ref{Jordan}~$2)~(ii)$, when $n=2p$ is even, are those of \textit{normal Bernstein algebras} of type $(2p,1)$. Normal Bernstein algebras are defined by equation $x^2y=\w(x)xy$. These are Bernstein-Jordan algebras, caracterized by the train equation $x^3-\w(x)x^2=0$ \cite{Ouat,Buse}.
	
	3) Multiplication tables in Theorem~\ref{2TP}~$2)~(iii)$ and in Corollary~\ref{Jordan}~$2)~(iii)$, when $n=2p$ is even, are those of the other class of \textit{train algebras of rank $3$ which are Jordan algebras} of type $(2p,1)$. They are defined by equation $x^3-2\w(x)x^2+\w(x)^2x=0$ \cite[Theorem~2.1]{Ouat}.
\end{note}
	
    4) Multiplication tables in Theorem~\ref{2TP}~$1)~(i')$ and $2)~(i')$ are those of \textit{train algebras} satisfying $x^3-\frac{3}{2}\w(x)x^2+\frac{1}{2}\w(x)^2x=0$. These are Lie triple algebras because of \cite[Proposition~5.2, (iii)]{BKO1}.

\end{document}